\documentclass[letterpaper, 10 pt, conference]{ieeeconf}
\IEEEoverridecommandlockouts
\usepackage{cite}
\usepackage{amsmath,amssymb,amsfonts}
\usepackage{algorithm}
\usepackage{algorithmic}
\usepackage{graphicx}
\usepackage{hyperref}
\usepackage{textcomp}
\usepackage{booktabs}
\usepackage[dvipsnames]{xcolor}
\usepackage{dblfloatfix}
\usepackage{caption}

\newtheorem{definition}{Definition}[section]
\newtheorem{assumption}{Assumption}[section]
\newtheorem{proposition}{Proposition}[section]
\newtheorem{remark}{Remark}[section]

\newcommand{\R}{\mathbb{R}}
\newcommand{\Spp}{\mathbb{S}_{++}}
\newcommand{\Spsd}{\mathbb{S}_{+}}
\newcommand{\tr}{\operatorname{tr}}
\newcommand{\blkdiag}{\operatorname{blkdiag}}
\newcommand{\norm}[1]{\left\lVert#1\right\rVert}
\newcommand{\vecop}{\operatorname{vec}}

\title{\LARGE \bf
The Separation Principle and the Dual–Certainty Equivalence Gap in Model Predictive Control
}

\author{Tren Baltussen, Nathan P. Lawrence, Alexander Katriniok, Ali Mesbah, Maurice Heemels 
\thanks{T. Baltussen, A. Katriniok and M. Heemels are with the Control Systems Technology Section, Eindhoven University of Technology, the Netherlands. N.P. Lawrence and A. Mesbah are with the Department of Chemical and Biomolecular
Engineering, UC Berkeley, California, USA. Email:
{\tt t.m.j.t.baltussen@tue.nl}}}%

\begin{document}
\maketitle
\pagestyle{plain}

\begin{abstract}
Dual control addresses the trade-off between \text{exploitation} and exploration, where control inputs both regulate the system and generate informative data for estimation and identification. For certain problem classes, control and estimation can be designed independently without loss of optimality, a property known as the separation principle. 
However, in stochastic control problems with model uncertainty and constraints, this principle generally breaks down, and introduces the need for dual control.
In this paper, we propose an information-weighted dual model predictive control (MPC) formulation and introduce metrics that quantify the dependence of the MPC policy on the uncertainty. 
We focus on parametric uncertainty in linear systems with Gaussian noise, though the metrics can be applied more broadly.
Numerical results show that the dependence of the MPC policy on the posterior covariance is largest under high uncertainty and vanishes as the posterior covariance contracts, providing empirical evidence of the dual effect in closed loop. Moreover, the dual controller improves regulation performance and model accuracy compared to certainty-equivalent MPC.
\end{abstract}

\section{Introduction}
The separation principle is a cornerstone of stochastic control, where controller design and state estimation can be decoupled without loss of optimality. The classical linear--quadratic--Gaussian (LQG) setting with known dynamics is a prime example, yielding a Kalman filter for estimation and a linear-quadratic regulator (LQR) for control \cite{BarShalomTse1974}.
However, in the presence of model uncertainty and/or constraints, the separation principle typically breaks down, and the optimal control law depends on the state as well as its distribution, i.e., its uncertainty \cite{Feldbaum60,BarShalomTse1974}. In general, the optimal controller must balance regulation and exploration, trading off immediate performance and information gain to improve future control decisions. This coupling between control and uncertainty is known as the dual effect \cite{BarShalomTse1974}.

Model predictive control (MPC) provides a practical framework for constrained control \cite{Mayne00,RawlingsMayne17}, and dual MPC variants attempt to approximate the dual effect with tractable approximations \cite{Heirung17,arcari_dual_2020a,baltussen2025}.
Many such methods quantify and compare regulation and learning performance, but fewer works explicitly quantify the structural dependence between control and uncertainty as an empirically measurable object.
As MPC is amenable to accommodating dual control, it is important to understand the separation property in the context of dual MPC.
Since the MPC control law is typically computed via numerical optimization, the structural coupling between control and uncertainty is not immediately clear. We aim to make this coupling quantitatively observable by defining a \emph{separation gap} and a covariance sensitivity metric.

\subsection{Related Work}
Several works address dual control using MPC.
Various methods are based on or related to approximate dynamic programming.
For example, \cite{arcari_dual_2020a} and \cite{HUSHARP2024} use scenario trees in MPC to address the trade-off between exploration and exploitation. 
This exploration and exploitation trade-off is then \textit{implicitly} governed by the cost function \cite{mesbah_stochastic_2018}.
Alternatively, \cite{Heirung17} reformulates a single input, single output dual control problem by minimizing the output error at the next time step.
In \cite{baltussen2025}, the authors propose a dual MPC that uses Gaussian process regression that exploit the dual effect to actively learn nonparametric uncertainties online. In addition, they provide robust recursive feasibility guarantees in the presence of uncertainty.
Similar to the work in this paper, \cite{soloperto_dual_2019} proposes a dual MPC formulation for linear systems with parametric uncertainty, where the exploration is driven by a covariance-weighted stage cost. Under bounded uncertainty, they establish practical stability and recursive feasibility of the MPC.
For an introduction to dual control, we refer the interested reader to \cite{BarShalomTse1974}. For a comprehensive overview of dual MPC, we refer the reader to the survey \cite{Mesbah2016}.

\subsection{Contributions}
In this paper, we study an information-weighted dual MPC formulation that incorporates a covariance-dependent quadratic surrogate of the information gain into the stage cost, combined with Bayesian linear regression of the model parameters. In this setting, we aim to quantitatively characterize the separation property in dual MPC.
Our contribution is an analysis framework inspired by the separation property.
We define quantitative metrics that measure control--uncertainty coupling in dual MPC, and complement this with metrics that measure identification quality and closed-loop performance.
Our contributions are summarized as follows:
\begin{itemize}
\item We propose an information-weighted dual MPC.
\item We introduce two \emph{separation metrics}: separation gap $S_t$ and covariance sensitivity $G_t$, and use them to quantify the control--uncertainty coupling.
\item We analyze the dual MPC with these separation metrics in Monte Carlo simulations for a linear system with parametric uncertainty subject to Gaussian noise, and compare its closed-loop regulation and identification performance against a certainty-equivalent MPC.
\end{itemize}

\section{Preliminaries}
\label{seq:prelim}
We will first recall some fundamental concepts in stochastic control theory that are relevant to our analysis of dual MPC, including the certainty-equivalence property, the separation property, and the dual effect. These concepts primarily originate from the work of Bar-Shalom and Tse \cite{BarShalomTse1974}.
\subsection{Certainty Equivalence and the Separation Property}
For a control problem, it is said that the \textit{certainty equivalence} property holds if the optimal control law can be obtained by solving the deterministic optimal control problem with the expected state or parameter estimates, without accounting for uncertainty in the state measurement or the system dynamics.
Then, the optimal control law coincides with the deterministic one evaluated at the current expected state (mean) and depends only on available observations, without accounting for the impact of current actions on future information. The LQG problem with known dynamics and additive Gaussian noise is a classical example of this \cite{BarShalomTse1974}.

Although the certainty equivalence property is a strong structural condition that holds only for specific problem classes, a \emph{certainty-equivalent controller} is often used as a heuristic in stochastic control.
The certainty-equivalent (CE) controller is obtained by solving the deterministic optimal control problem and replacing the unknown states or parameters with their current expectation, as if it were certain, thereby ignoring uncertainty in the control design.
When the CE controller is truly an optimal solution to the optimal control problem, it is said that \textit{certainty equivalence property holds} \cite{BarShalomTse1974}.


The \textit{separation property} is a weaker property that states that the optimal control law depends on the current state estimate, and is invariant to the distribution (e.g., covariance) in the belief. 
In other words, the optimal control law is separation-invariant if it does not explicitly depend on the uncertainty state. When the separation property holds, the control design can be decoupled from estimation without loss of optimality, even if certainty equivalence does not hold.
However, unlike certainty equivalence, for the separation property to hold, the control law need not coincide with the deterministic optimal control law and may still depend on the current state estimate \cite{BarShalomTse1974}.


\subsection{Dual Effect}
When the separation property does not hold, the optimal control law depends on the uncertainty state, and the controller must balance regulation and exploration, a problem known as dual control \cite{Feldbaum60}.
In contrast to the separation property, an optimal dual controller explicitly accounts for the effect of current inputs on future information. Such a controller trades off immediate control performance with information gain that improves long-term performance \cite{BarShalomTse1974}.

\begin{definition}\label{def:dual_effect}
A controller exhibits a \emph{dual effect} if the chosen control action depends on the state or system uncertainty estimate and its distribution (e.g. covariance) such that it can influence future information through the belief update dynamics.
\end{definition}

As MPC is widely used to approximate optimal control laws in constrained stochastic optimal control problems, it is of interest to understand whether, and to what extent, this separation holds for dual MPC formulations, and how this relates to the dual effect.

\section{Problem setup and Bayesian belief update}
\subsection{Uncertain Linear Dynamics}
We consider the discrete-time stochastic linear system
\begin{equation}
x_{t+1}=A^\star x_t + B^\star u_t + w_t,\qquad w_t\sim\mathcal{N}(0,\sigma_w^2 I_n),
\label{eq:true_sys}
\end{equation}
where $x_t\in\R^n$ is the state, $u_t\in\R^m$ is the control input, $w_t\in\R^n$ is i.i.d.\ process noise with known variance $\sigma_w^2>0$, and $I_n$ denotes the $n\times n$ identity matrix.
The system matrices $A^\star\in\R^{n\times n}$ and $B^\star\in\R^{n\times m}$ are unknown and collected as
\begin{equation}
\Theta^\star=[A^\star\;B^\star]\in\R^{n\times(n+m)},\qquad \theta^\star=\vecop(\Theta^\star)\in\R^{n(n+m)},
\end{equation}
where $\vecop(\cdot)$ denotes column-wise vectorization.
Let $z_t=[x_t^\top\;u_t^\top]^\top\in\R^{n+m}$ denote the joint state--input vector and $\Phi_t=z_t^\top\otimes I_n\in\R^{n\times n(n+m)}$ the associated regression matrix, where $\otimes$ denotes the Kronecker product.
Then
\begin{equation}
x_{t+1}=\Phi_t \theta^\star + w_t.
\label{eq:regression}
\end{equation}
We aim to stabilize the system to the origin and consider a quadratic regulation cost
\begin{equation}
    \ell_{\mathrm{reg}}(x,u)= x^\top Q x + u^\top R u,
\label{eq:ell_reg}
\end{equation}
where $Q\in\Spp^n$ and $R\in\Spp^m$.
We denote by $\Spp^d$, $\Spsd^d$ the sets of $d\times d$ symmetric positive definite and positive semidefinite matrices, respectively.
We consider input constraints $u_t\in\mathcal{U}\subseteq\R^m$ and assume that the state is fully observed at each time step. State constraints are not considered in this work to focus on dual effect induced by the cost shaping. However, the proposed metrics and analysis framework can be applied to (soft- or chance-)constrained settings as well.

\subsection{Bayesian Linear Regression}
The parameter estimation is initialized with a Gaussian prior $\theta_0\sim\mathcal{N}(\hat\theta_0,\Sigma_0)$, where $\hat\theta_0$ is the prior mean and $\Sigma_0$ is the prior covariance.
We maintain a Gaussian posterior $\theta_t \sim\mathcal{N}(\hat\theta_t,\Sigma_t)$, where $\hat\theta_t\in\R^{n(n+m)}$ is the posterior mean and $\Sigma_t\in\Spsd^{n(n+m)}$ is the posterior covariance, and update it at every time step using the covariance-form recursion \cite{goodwin1984adaptive}
\begin{align}
N_t &= \sigma_{w}^2 I_n + \Phi_t\Sigma_t\Phi_t^\top,\label{eq:S}\\
K_t &= \Sigma_t \Phi_t^\top N_t^{-1},\label{eq:K}\\
\hat\theta_{t+1} &= \hat\theta_t + K_t(x_{t+1}-\Phi_t\hat\theta_t),\label{eq:theta_update}\\
\Sigma_{t+1} &= \Sigma_t - K_t\Phi_t\Sigma_t.\label{eq:Sigma_update}
\end{align}
where $N_t\in\R^{n\times n}$ is the innovation covariance, $K_t\in\R^{n(n+m)\times n}$ is the Kalman gain, and $\sigma_{w}^2>0$ is the assumed filter noise variance. Note that the $\hat{\theta}_{t+1}$ and $\Sigma_{t+1}$ depend on the current input $u_t$ through $\Phi_t$.
We select the current model estimate as the posterior mean:
\begin{equation}
\hat\Theta_t:=\mathrm{mat}(\hat\theta_t)=[\hat A_t\;\hat B_t],
\end{equation}
where $\mathrm{mat}(\cdot)$ denotes the inverse of $\vecop(\cdot)$, reshaping the parameter vector back into a matrix.

\section{Model Predictive Control}
Before formulating the separation metrics, we introduce three MPC variants.
Firstly, a certainty-equivalent MPC (CE-MPC) that ignores the posterior covariance and optimizes control inputs based on the current model estimate $\hat\theta_t$. The dual MPC with an information-weighted stage cost that explicitly depends on the posterior covariance $\Sigma_t$. Finally, we consider an oracle MPC with perfect system information.




%

\subsection{Certainty-Equivalent MPC}
Given the current model estimate $(\hat A_t,\hat B_t)$, the CE-MPC solves the following finite-horizon quadratic program (QP) over a prediction horizon of $N$ steps:
\begin{equation}
\begin{aligned}
\min_{\{u_{k|t}\}}~~&\sum_{k=0}^{N-1} x_{k|t}^\top Q x_{k|t}+u_{k|t}^\top R u_{k|t} +x_{N|t}^\top P x_{N|t}\\
\text{s.t.}~~&x_{0|t}=x_t,\quad x_{k+1|t}=\hat A_t x_{k|t}+\hat B_t u_{k|t},\\
& \quad u_{\min}\le u_{k|t}\le u_{\max},
\end{aligned}
\label{eq:ce_mpc}
\end{equation}
where $x_{k|t}\in\R^n$ and $u_{k|t}\in\R^m$ denote the predicted state and input at prediction step $k$ from current time $t$, and $u_{\min},u_{\max}\in\R^m$ define the input constraints.
The controller applies the first optimal input $u_t=u^\star_{0|t}$ in a receding-horizon fashion. In this paper we fix the terminal cost matrix $P=Q$.

\subsection{Approximate Information Gain}
We now extend the CE-MPC in \eqref{eq:ce_mpc} by explicitly rewarding informative control actions, incorporating the essential exploration--exploitation tradeoff.
A standard measure of information is the log-determinant of the Fisher information matrix \cite{cover2006elements},
\begin{equation}
\mathcal{I}(\Sigma) := \log \det (\Sigma^{-1}).
\end{equation}
The one-step information gain is defined as
\begin{equation}
\Delta \mathcal{I}(z_t,\Sigma_t)
=
\log \det(\Sigma_{t+1}^{-1})
-
\log \det(\Sigma_t^{-1}).
\label{eq:exact_info_gain}
\end{equation}

In information form, the covariance update \eqref{eq:Sigma_update} is given by
\begin{equation}
\Sigma_{t+1}^{-1}
=
\Sigma_t^{-1}
+
\tfrac{1}{\sigma_w^2}(z_t z_t^\top \otimes I_n).
\label{eq:information_update}
\end{equation}

Using \eqref{eq:information_update}, we obtain
\begin{equation}
\Delta \mathcal{I}(z_t,\Sigma_t)
=
\log\det\!\left(
I + \tfrac{1}{\sigma_w^2}\Sigma_t (z_t z_t^\top \otimes I_n)
\right).
\end{equation}

To preserve a quadratic MPC structure, we use the first-order approximation \cite{boyd2004convex}
\begin{equation}
\log\det(I+X) \approx \operatorname{tr}(X),
\end{equation}
which yields
\begin{equation}
\Delta \mathcal{I}(z_t,\Sigma_t)
\approx
\tfrac{1}{\sigma_w^2}\operatorname{tr}\!\big(\Sigma_t (z_t z_t^\top \otimes I_n)\big).
\label{eq:approx_info_gain}
\end{equation}

Since the right-hand side is quadratic in \(z_t\), there exists a unique matrix \(W(\Sigma_t)\in\mathbb{S}_+^{n+m}\) such that
\begin{equation}
z_t^\top W(\Sigma_t) z_t
=
\tfrac{1}{\sigma_w^2}\operatorname{tr}\!\big(\Sigma_t (z_t z_t^\top \otimes I_n)\big).
\label{eq:W_definition}
\end{equation}

Partitioning \(\Sigma_t\) into \((n+m)\times(n+m)\) blocks
\[ {\scriptsize
\Sigma_t =
\begin{bmatrix}
\Sigma_{11} & \cdots & \Sigma_{1,n+m}\\
\vdots & \ddots & \vdots\\
\Sigma_{n+m,1} & \cdots & \Sigma_{n+m,n+m} 
\end{bmatrix},
\qquad
\Sigma_{ij}\in\mathbb{R}^{n\times n},}
\]
the entries of \(W(\Sigma_t)\) are given by
\begin{equation}
    \label{eq:W}
[W(\Sigma_t)]_{ij}
=
\tfrac{1}{\sigma_w^2}\operatorname{tr}(\Sigma_{ij}).
\vspace{-1mm}
\end{equation}
This yields a quadratic approximation of the information gain that can be embedded directly into the MPC stage cost.

\subsection{Information-Weighted Dual Stage Cost}
We formulate an information-weighted stage cost by augmenting the regulation cost $\ell_{\mathrm{reg}}$ in \eqref{eq:ell_reg} with a covariance-dependent quadratic term.
Based on first-order approximations of the $\log\det$ information gain in \eqref{eq:approx_info_gain}, we define
\begin{equation}
\ell_{\mathrm{dual}}(x,u,\Sigma)=x^\top Q x + u^\top R u - \alpha\, z^\top W(\Sigma)z,
\label{eq:ldual}
\end{equation}
where $z^\top:=\begin{bmatrix} x^\top & u^\top \end{bmatrix}$, and $\alpha\ge 0$ is a scalar exploration weight that trades off regulation against information gain.
We define the information-weighted stage-cost matrix as
\begin{equation}
L(\Sigma) :=\blkdiag(Q,R)-\alpha W(\Sigma)\in \mathbb{S}_+^{n+m}.
\label{eq:L}
\end{equation}
Then, the stage cost can be written as
\begin{equation}
 \ell_{\mathrm{dual}}(x,u,\Sigma)=z^\top L(\Sigma) z.
\end{equation}

\subsection{Information-Weighted Dual MPC Formulation}
The information-weighted dual MPC uses the current belief $(\hat\theta_t,\Sigma_t)$ to shape the stage cost. We fix $\Sigma_t$ over the prediction horizon to isolate the effect of the current covariance on the MPC. The resulting QP reads as follows:
\begin{equation}
\begin{aligned}
\min_{\{u_{k|t}\}}~~&\sum_{k=0}^{N-1}z_{k|t}^\top L(\Sigma_t)z_{k|t}+x_{N|t}^\top P x_{N|t}\\
\text{s.t.}~~&x_{0|t}=x_t,\quad x_{k+1|t}=\hat A_t x_{k|t}+\hat B_t u_{k|t},\\
&u_{\min}\le u_{k|t}\le u_{\max}.
\end{aligned}
\label{eq:dual_mpc}
\end{equation}
We require $L(\Sigma_t)\succ 0$ to ensure strict convexity of the resulting QP; this can always be achieved by choosing $\alpha$ sufficiently small with respect to $(Q,R)$, $W(\Sigma)$ and the prior covariance $\Sigma_0$, assuming that $\Sigma_{t+1} \preceq \Sigma_t$ for all $t$.
After applying $u_t$, the belief is updated via \eqref{eq:theta_update}--\eqref{eq:Sigma_update}.
This creates an explicit dependence of $u_t$ on $\Sigma_t$, and of $\Sigma_{t+1}$ on $u_t$.

\begin{remark}
Note that this information-weighted dual MPC does not explicitly propagate the covariance over the horizon, but rather uses the current $\Sigma_t$ to shape the stage cost at all prediction steps, following Definition \ref{def:dual_effect}. Therefore, this MPC does not feature the dual effect as formally defined in \cite{BarShalomTse1974}, which requires that the control input can influence future information through the belief update dynamics over the MPC horizon, causally anticipating the future covariance \cite{baltussen2025}. This  is done, for example, in \textit{wide-sense control} \cite{TseBarShalom1973}. For the sake of simplicity, we confine our analysis to this static covariance shaping, which already induces a structural coupling between control and uncertainty.
Still, the explicit dependence of the stage cost on $\Sigma_t$ creates a structural coupling between control and uncertainty that can be observed through the separation metrics defined in the next section.
\end{remark}

\subsection{Oracle MPC}
For the purpose of evaluating the separation and validation metrics, we also consider an oracle MPC which solves \eqref{eq:ce_mpc} with the true system parameters $\theta^\star$, i.e. $u_t^{\mathrm{orc}}:=\pi_{\mathrm{CE}}(x_t,\theta^\star)$.

\section{Separation and Validation Metrics}
\subsection{Separation gap}
\label{sec:sep_gap}
In order to quantify the effect of including the information-weighted cost term on the control law, we define the separation gap as the Euclidean distance between the dual control input and the certainty-equivalent input at the same \emph{belief state} $(x_t,\hat\theta_t,\Sigma_t)$.
Let us denote the first MPC inputs from \eqref{eq:ce_mpc} and \eqref{eq:dual_mpc}, solved with the same $(x_t,\hat\theta_t,\Sigma_t)$, as
\begin{equation}
u_t^{\mathrm{CE}}:=\pi_{\mathrm{CE}}(x_t,\hat\theta_t),
\qquad
u_t^{\mathrm{dual}}:=\pi_{\mathrm{dual}}(x_t,\hat\theta_t,\Sigma_t).
\label{eq:shadow_controls}
\end{equation}
We define the separation gap as follows:
\begin{equation}
S_t := \norm{u_t^{\mathrm{dual}}-u_t^{\mathrm{CE}}}_2,
\label{eq:St}
\end{equation}
where $u_t^{\mathrm{CE}}$ and $u_t^{\mathrm{dual}}$ are the first MPC inputs from \eqref{eq:ce_mpc} and \eqref{eq:dual_mpc}, respectively, and, $\norm{\cdot}_2$ denotes the Euclidean norm. In order to isolate the structural effect of covariance on the control law, we perform experiments with the oracle MPC, and we evaluate both the CE and dual MPC problems, at the same belief state $(x_t,\hat\theta_t,\Sigma_t)$, as defined in \eqref{eq:shadow_controls}.
This \emph{shadow} approach yields paired inputs from which we compute the separation metrics and additional validation quantities.

\paragraph*{Interpretation}
$S_t$ is a direct measure of the dependence on the posterior covariance. Namely, $S_t$ indicates how strong $\Sigma_t$ changes the certainty-equivalent control law. Conversely, if the dependence vanishes, then $u_t^{\mathrm{dual}}=u_t^{\mathrm{CE}}$ and $S_t=0$.

\subsection{Covariance sensitivity}
We define a finite-difference approximation of the sensitivity of the dual control law to the posterior covariance:
\begin{equation}
G_t :=
\frac{\norm{\pi_{\mathrm{dual}}(x_t,\hat\theta_t,(1+\varepsilon)\Sigma_t)
-\pi_{\mathrm{dual}}(x_t,\hat\theta_t,\Sigma_t)}_2}
{\varepsilon \norm{\Sigma_t}_F },
\label{eq:Gt}
\end{equation}
where $\varepsilon>0$ is a small perturbation parameter and $\norm{\cdot}_F$ denotes the Frobenius norm.
The normalization ensures that $G_t$ is comparable across time steps and simulation runs. In our experiments, we take $\varepsilon=0.05$.

\paragraph*{Interpretation}
$G_t$ captures the local dependence of the control law on the magnitude of parametric uncertainty.
When the QP optimizer is unique and differentiable with respect to the problem data, $G_t$ approximates a directional derivative of $\pi^{\mathrm{dual}}$ along the ray $\Sigma\mapsto (1+\varepsilon)\Sigma$.

\begin{remark}
Note that the separation gap $S_t$ is a property of the MPC policy and does not require the MPC to be optimal. Rather, it is inspired by the separation principle, but it can be evaluated for any control law. In particular, it can be of interest for quantify and understand how dual MPC control laws behave under, and depend on the uncertainty.
\end{remark}

\subsection{Validation Metrics}
In addition to the separation-theoretic metrics above, we record two validation quantities that quantify identification quality and control performance.

\paragraph*{Model Error}
We measure the global model error using the Frobenius norm of the parameter estimation error:
\begin{equation}E_t^{\mathrm{par}} := \norm{\hat\Theta_t - \Theta^\star}_F.
\label{eq:Ecr}
\end{equation}

\paragraph*{Oracle Mismatch}
Let $u_t^{\mathrm{orc}}:=\pi_{\mathrm{CE}}(x_t,\theta^\star)$ denote oracle MPC. Then, the \emph{oracle mismatch}
\begin{equation}
M_t^{\mathrm{orc}} := \norm{u_t - u_t^{\mathrm{orc}}}_2
\label{eq:Morc}
\end{equation}
measures the deviation of the actually applied input $u_t$ from the input that a fully informed CE controller would select.
By construction, $M_t^{\mathrm{orc}}=0$ for $u_t=u_t^{\mathrm{orc}}$.

\subsection{When Does Separation Hold?}
In the context of our information-weighted dual MPC, we can establish structural conditions under which this separation gap is positive, and characterize the sensitivity of the MPC control law on the belief state.

\begin{assumption}\label{ass:standing}
Throughout, we assume:
(i) $(A,B)$ is stabilizable and $(A,Q^{1/2})$ is detectable;
(ii) $Q\succ0$, $R\succ0$, and $P\succ0$;
(iii) $L(\Sigma_t)\succ0$ for all admissible $\Sigma_t$;
(iv) the QPs \eqref{eq:ce_mpc}--\eqref{eq:dual_mpc} are feasible and admit unique minimizers.
\end{assumption}

\begin{assumption}[Local smoothness for sensitivity]\label{ass:smooth}
For each time $t$, the linear independence constraint qualification (LICQ) holds at the optimizer of \eqref{eq:dual_mpc}.
\end{assumption}


\begin{proposition}[Lack of separation]\label{prop:break}
Fix a time $t$ and suppose $\alpha>0$ and $W(\Sigma_t)\neq 0$.
Then the objective in \eqref{eq:dual_mpc} depends explicitly on $\Sigma_t$ through $L(\Sigma_t)$, and hence the dual control law is, in general, covariance-dependent.
In particular, if the covariance-dependent term changes the optimizer relative to \eqref{eq:ce_mpc}, then $u_t^{\mathrm{dual}}\neq u_t^{\mathrm{CE}}$ and $S_t>0$.
\end{proposition}

\begin{proof}
Since $\alpha > 0$ and $W(\Sigma_t) \neq 0$, the objective
in \eqref{eq:dual_mpc} depends explicitly on $\Sigma_t$ through $L(\Sigma_t)$.
Hence, the optimizer $u_t^{\mathrm{dual}}$ is, in general, a function of $\Sigma_t$,
which shows dependence on $\Sigma_t$. 
If the covariance-dependent term changes the optimizer relative to
\eqref{eq:ce_mpc}, then $u_t^{\mathrm{dual}} \neq u_t^{\mathrm{CE}}$, which implies
$S_t > 0$.
\end{proof}


\begin{proposition}[Dependence on covariance]\label{prop:lipschitz}
Under Assumptions~\ref{ass:standing} and \ref{ass:smooth}, and provided the active set of
\eqref{eq:dual_mpc} remains unchanged in a neighborhood of $\Sigma_t$, the finite-difference quantity $G_t$ in \eqref{eq:Gt} converges to the
directional derivative of this mapping along $\Sigma \mapsto (1+\varepsilon)\Sigma$
as $\varepsilon\to 0$, whenever this directional derivative exists.
\end{proposition}

\begin{proof}
With a constant active set and LICQ, the minimizer is an affine function of the QP coefficients in a neighborhood of $\Sigma_t$.
As the QP coefficients depend smoothly on $\Sigma$, the mapping $\Sigma \mapsto u^{\mathrm{dual}}(x_t,\hat\theta_t,\Sigma)$ is locally Lipschitz around $\Sigma_t$.
Therefore, $\Sigma\mapsto u^{\mathrm{dual}}$ is locally Lipschitz and differentiable in a neighborhood of $\Sigma_t$; the finite difference in \eqref{eq:Gt} estimates the corresponding directional derivative.
\end{proof}

\begin{remark}[Connection to the dual effect]
Propositions~\ref{prop:break}--\ref{prop:lipschitz} formalize the mechanism by which our dual MPC breaks separation: $\Sigma_t$ enters the control law explicitly via cost shaping.
This structural coupling is precisely is absent in certainty-equivalent designs.
\end{remark}

\section{Numerical results}
\subsection{Problem Setup}
We consider a discrete-time double integrator with sampling time $T_s=0.1$:
\begin{equation}
A^\star=\begin{bmatrix}1&T_s\\0&1\end{bmatrix},\qquad
B^\star=\begin{bmatrix}0.5T_s^2\\T_s\end{bmatrix},
\end{equation}
which yields the following true dynamics:
\begin{equation}x_{t+1}=\begin{bmatrix}1&0.1\\0&1\end{bmatrix}x_t+\begin{bmatrix}0.005\\0.1\end{bmatrix}u_t+w_t.\end{equation}

The covariance of the Gaussian noise are $\sigma_{w}^2=5\cdot10^{-4}$.
We set $Q=\mathrm{diag}(10,1)$, $R=1$ and $P=Q$.
We set $N=3$ and simulate $n_{\mathrm{sim}}=100$ time steps from $x_0=[0.4 \;0.1]^\top$.
The input is constrained to $u\in[-10,10]$.

\begin{algorithm}[t]
\caption{Closed-Loop Evaluation of Dual and CE-MPC}\label{alg:sep}
\begin{algorithmic}[1]
\STATE \textbf{Given:} belief $(\hat\theta_t,\Sigma_t)$, state $x_t$, horizon $N$, cost matrices $(Q,R,P)$, parameters ($\alpha,\varepsilon$)
\STATE Extract $(\hat A_t,\hat B_t)=\mathrm{mat}(\hat\theta_t)$
\STATE \textbf{CE shadow:} solve \eqref{eq:ce_mpc} to obtain $u_t^{\mathrm{CE}}$
\STATE Compute $W(\Sigma_t)$ via \eqref{eq:W} and form $L(\Sigma_t)$ via \eqref{eq:L}
\STATE \textbf{Dual shadow:} solve \eqref{eq:dual_mpc} to obtain $u_t^{\mathrm{dual}}$
\STATE Compute $S_t$ via \eqref{eq:St} and $G_t$ via \eqref{eq:Gt}
\STATE \textbf{Apply:} select $u_t$ (dual, CE, or oracle); propagate \eqref{eq:true_sys} to $x_{t+1}$
\STATE Compute metrics $E_t^{\mathrm{par}}$, $M_t^{\mathrm{orc}}$ via \eqref{eq:Ecr}--\eqref{eq:Morc}
\STATE \textbf{Belief update:} update $(\hat\theta_{t+1},\Sigma_{t+1})$ via \eqref{eq:theta_update}--\eqref{eq:Sigma_update}
\end{algorithmic}
\end{algorithm}

\paragraph*{Initialization}
We parameterize $\Theta^\star=[A^\star\;B^\star]$ and initialize a biased mean
\begin{equation}
\hat A_0=A^\star+\begin{bmatrix}0.5&0.5\\0&0.25\end{bmatrix},\qquad
\hat B_0=B^\star+\begin{bmatrix}0.1\\0.25\end{bmatrix},
\end{equation}
with $\Sigma_0=1I_6$.
The exploration weight is set to $\alpha=1$ and $\varepsilon=0.05$ for the finite-difference sensitivity \eqref{eq:Gt}.

\paragraph*{Monte Carlo Experiments}
We perform 20 Monte Carlo episodes with independent noise realizations that are identical across methods and compare:
(i) dual MPC $u_t^{\mathrm{dual}}(x_t, \hat{\theta}_t, \Sigma_t)$, (ii) CE-MPC $u_t^{\mathrm{CE}}(x_t, \hat{\theta}_t)$, (iii) oracle MPC $u_t^{\mathrm{orc}}(x_t, \theta^\star)$.
We record:
(a)~the cumulative regulation cost $J^{\mathrm{reg}}_t = \sum_{i=0}^{t}\ell_{\mathrm{reg}(x_i, u_i)}$,
(b)~posterior uncertainty measured by the covariance trace $\tr(\Sigma_t)$,
(c)~separation metrics $S_t$ and $G_t$, and
(d)~validation metrics $E_t^{\mathrm{par}}$ and $M_t^{\mathrm{orc}}$.
Algorithm~\ref{alg:sep} summarizes the closed-loop procedure, including the shadow solves used to compute separation and validation metrics at each time step, detailed in Section~\ref{sec:sep_gap}.

\paragraph*{Post-Learning Evaluation}
After the learning phase, we perform post-learning Monte Carlo evaluations where both CE and dual MPC use the same cost function ($\alpha=0$) and fully exploit their identified model in a certainty-equivalent manner. In these simulations, the model is not updated anymore, and the controllers operate with the final model estimates obtained from their respective learning phases, isolating the long-term benefits of the dual effect during the learning phase on the subsequent exploitation performance.

\subsection{Discussion}
The Monte Carlo results are shown in Fig. \ref{fig:mc} and Fig. \ref{fig:sep_time}.
The initial excitation of the dual MPC is dominated by the large covariance trace. After initial excitation, the covariance reduces and the separation gap reduces. As the covariance reduces further, the sensitivity to the covariance increases again, which is followed by an increased separation gap (around $t=4$ seconds).
It is expected that this induces more exploration, leading to further reduction in posterior covariance. Meanwhile, the CE-MPC converges slower than the dual MPC and settles to a fixed band.
The separation metrics reveal \emph{how} this happens structurally: $S_t$ and $G_t$ are largest when $\Sigma_t$ is large and diminish as $\Sigma_t$ contracts.
In other words, the dual effect becomes directly visible as a time-varying dependence of the control law on covariance of the system parameters.
It is important to note, that in this problem setting, there is a strong correlation between the separation gap and the covariance trace. However, the separation gap is a general metric that extends beyond this problem setting and can be used to quantify the dependence on the uncertainty distribution in more general settings.

\begin{figure}[t]
\centering
\includegraphics[width=1\linewidth]{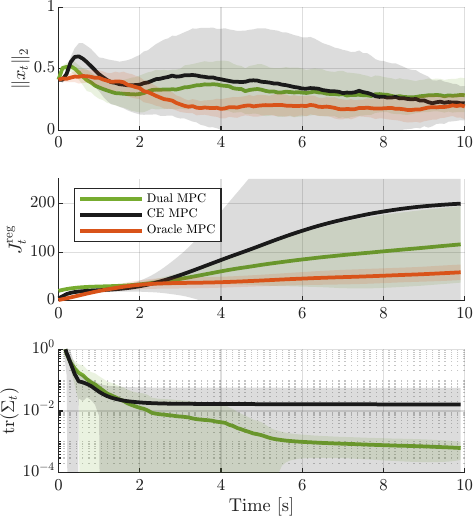}
\caption{The Monte Carlo experiments of dual MPC, CE-MPC, and oracle MPC. Here, the solid lines indicate the mean and the shaded areas indicate one standard deviation.}
\label{fig:mc}
\vspace{-4mm}
\end{figure}

\begin{figure}[t]
\centering
\includegraphics[width=1\linewidth]{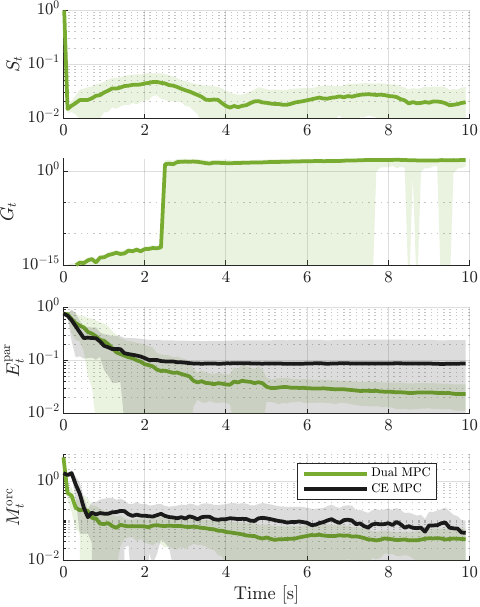}
\caption{The separation gap $S_t$, covariance sensitivity $G_t$, model error $E_t^{\mathrm{par}}$ and oracle mismatch $M_t^{\mathrm{orc}}$ in Monte Carlo experiments.}
\label{fig:sep_time}
\vspace{-4mm}
\end{figure}

Although the initial excitation of the dual MPC leads to higher regulation cost than CE-MPC (up to $t=2$ seconds), it also leads to faster reduction in model error and posterior uncertainty, which in turn leads to better control performance during the exploitation phase (after $t=4$ seconds).
The dual MPC attains a lower cumulative regulation cost than CE-MPC.
The dual MPC reduces parameter uncertainty and model error faster than CE-MPC, consistent with the intended excitation effect of the information-weighted cost. Overall, the dual MPC achieves lower regulation cost and oracle mismatch. Note that we can observe a correlation between the separation gap and the oracle mismatch of the dual MPC. This suggests that the separation gap is indicative of certainty equivalence principle, discussed in Section \ref{seq:prelim}.

The results of the post-learning evaluation experiment are shown in Fig. \ref{fig:alpha}. Here, both CE and dual MPC have the same cost function (i.e., $\alpha=0$) and operate in a certainty-equivalent manner, but they use different model estimates obtained from their respective learning phases. The post-learning evaluation confirms that the identified model under dual MPC is more accurate and leads to better closed-loop performance even when both controllers use the same cost function and operate in a certainty-equivalent manner, highlighting the long-term benefits of the dual effect during the learning phase. It is expected that this originates from the control-relevant excitation induced by the information-weighted cost, which is reflected in the separation metrics.

\section{Conclusion}
In this work, we studied an information-weighted dual MPC formulation through the lens of the separation principle. 
We introduced the \emph{separation gap} as a metric to empirically quantify the dependence of the (MPC) control law on the posterior uncertainty, and we defined a finite-difference approximation of the covariance sensitivity to capture the local dependence of the control law on the magnitude of the posterior uncertainty.
Numerical results show this dependence is most pronounced under high posterior uncertainty and diminishes as the belief concentrates, offering an empirical bridge between classical dual-effect theory and modern dual MPC formulations. Moreover, post-learning evaluations show that dual MPC can improve model estimates and closed-loop performance, even when both controllers subsequently operate in a certainty-equivalent manner, showing the potential benefits of dual control.

\begin{figure}[t]
\centering
\includegraphics[width=1\linewidth]{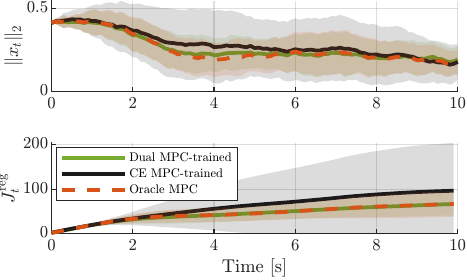}
\caption{Post-learning Monte Carlo evaluation. Note $\alpha = 0$.}
\label{fig:alpha}
\vspace{-4mm}
\end{figure}

Future work will explore the impact of updating the belief across the horizon, similar to the wide-sense control formulation in \cite{TseBarShalom1973}, which would more directly connect to the classical dual effect as defined in \cite{BarShalomTse1974}. In addition, we will study the impact of multi-step information propagation on the separation gap and assess the resulting trade-offs between performance and computational complexity.

\bibliographystyle{IEEEtran}
\bibliography{Ref}

@article{Feldbaum60,
  author  = {Feldbaum, A. A.},
  title   = {Dual Control Theory. {I}},
  journal = {Automation and Remote Control},
  year    = {1960},
  note    = {English translation of original work on dual control}
}

@article{Mayne00,
  author  = {Mayne, David Q. and Rawlings, James B. and Rao, Christopher V. and Scokaert, Pierre O. M.},
  title   = {Constrained Model Predictive Control: Stability and Optimality},
  journal = {Automatica},
  volume  = {36},
  number  = {6},
  pages   = {789--814},
  year    = {2000},
  doi     = {10.1016/S0005-1098(99)00214-9}
}

@book{RawlingsMayne17,
  author    = {Rawlings, James B. and Mayne, David Q. and Diehl, Moritz},
  title     = {Model Predictive Control: Theory, Computation, and Design},
  publisher = {Nob Hill Publishing},
  edition   = {2},
  year      = {2017},
}

@article{Heirung17,
  author  = {Heirung, Tor Aksel N. and Ydstie, Bjarne Erik and Foss, Bjarne A.},
  title   = {Dual Adaptive Model Predictive Control},
  journal = {Automatica},
  volume  = {80},
  pages   = {340--348},
  year    = {2017},
  doi     = {10.1016/j.automatica.2017.01.030}
}

@article{arcari_dual_2020a,
title = {An Approximate Dynamic Programming Approach for Dual Stochastic Model Predictive Control},
journal = {IFAC-PapersOnLine},
volume = {53},
number = {2},
pages = {8105-8111},
year = {2020},
issn = {2405-8963},
author = {Elena Arcari and Lukas Hewing and Melanie N. Zeilinger},
}

@article{TseBarShalom1973,
  author  = {E. Tse and Y. Bar-Shalom},
  title   = {An actively adaptive control for linear systems with random parameters via the dual control approach},
  journal = {IEEE Transactions on Automatic Control},
  volume  = {18},
  number  = {2},
  pages   = {109--117},
  year    = {1973}
}

@article{BarShalomTse1974,
  author  = {Y. Bar-Shalom and E. Tse},
  title   = {Dual effect, certainty equivalence, and separation in stochastic control},
  journal = {IEEE Transactions on Automatic Control},
  volume  = {19},
  number  = {5},
  pages   = {494--500},
  year    = {1974}
}

@article{Mesbah2016,
  author  = {A. Mesbah},
  title   = {Stochastic Model Predictive Control: An Overview and Perspectives for Future Research},
  journal = {IEEE Control Systems Magazine},
  volume  = {36},
  number  = {6},
  pages   = {30--44},
  year    = {2016}
}

@article{mesbah_stochastic_2018,
	title = {Stochastic model predictive control with active uncertainty learning: {A} {Survey} on dual control},
	volume = {45},
	issn = {13675788},
	shorttitle = {Stochastic model predictive control with active uncertainty learning},
	doi = {10.1016/j.arcontrol.2017.11.001},
	urldate = {2023-04-20},
	journal = {Annual Reviews in Control},
	author = {Mesbah, Ali},
	year = {2018},
	pages = {107--117},
}

@article{HUSHARP2024,
author = {Haimin Hu and David Isele and Sangjae Bae and Jaime F Fisac},
title ={Active uncertainty reduction for safe and efficient interaction planning: A shielding-aware dual control approach},
journal = {The International Journal of Robotics Research},
volume = {43},
number = {9},
pages = {1382-1408},
year = {2024},
doi = {10.1177/02783649231215371}
}

@inproceedings{soloperto_dual_2019,
  title={Dual adaptive {M}{P}{C} for output tracking of linear systems},
  author={Soloperto, Raffaele and K{\"o}hler, Johannes and M{\"u}ller, Matthias A and Allg{\"o}wer, Frank},
  booktitle={58th Conference on Decision and Control (CDC)},
  publisher = {IEEE},
  pages={1377--1382},
  year={2019},
}

@inproceedings{baltussen2025,
  title={Dual {M}{P}{C} for Active Learning of Nonparametric Uncertainties}, 
  author={Tren Baltussen and Maurice Heemels and Alexander Katriniok},
  year={2026},
  booktitle = {Accepted for European Control Conference},
  note = {{P}reprint available at arXiv:2511.08542},
}

@book{goodwin1984adaptive,
  title={Adaptive Filtering Prediction and Control},
  author={Goodwin, Graham C. and Sin, Kwai Sang},
  year={1984},
  publisher={Prentice-Hall}
}

@book{cover2006elements,
  title={Elements of Information Theory},
  author={Cover, Thomas M. and Thomas, Joy A.},
  year={2006},
  publisher={Wiley}
}

@book{boyd2004convex,
  title={Convex Optimization},
  author={Boyd, Stephen and Vandenberghe, Lieven},
  year={2004},
  publisher={Cambridge University Press}
}

\end{document}